\documentclass[a4paper,reqno]{amsart}

\usepackage{enumerate}
\usepackage{amssymb} 
\usepackage[colorlinks=true]{hyperref}
\hypersetup{urlcolor=blue, citecolor=cyan}
\hypersetup{citecolor=blue}

\numberwithin{equation}{section}

\newtheorem{thm}{Theorem}[section]
\newtheorem{lem}[thm]{Lemma}

\newtheorem{prop}[thm]{Proposition}
\theoremstyle{definition}
\newtheorem{rem}[thm]{Remark}

\newcommand\R{{\mathbb R}}
\newcommand\C{{\mathbb C}}
\newcommand\N{{\mathbb N}}

\newcommand\Tma{T_{\mathrm{max}}}
\newcommand\Tstar{T^\star }

\newcommand\Srn{{\mathcal S}(\R^N )}

\newcommand\Imqu{{m}}
\newcommand\Imqd{{n}}
\newcommand\Imqt{{k}}

\newcommand\Imqc{{j}}
\newcommand\Imqs{{\mu}}
\newcommand\Imqp{{\nu}}
\newcommand\Imqh{{J}}
\newcommand\Spa{{\mathcal X}}

\newcommand\CST[1]{C_{#1}}

\newcommand\MScN[1]{\href{http://www.ams.org/mathscinet-getitem?mr=#1}{\nolinkurl{(#1)}}}
\newcommand\DOI[1]{\href{http://dx.doi.org/#1}{(doi: \nolinkurl{#1})}}
\newcommand\LINK[1]{\href{#1}{(link: \nolinkurl{#1})}}

\newcommand\DIb{v_0 }

\newcommand\goto{\mathop{\longrightarrow}}

\newcommand\la{\langle}
\newcommand\ra{\rangle}

\begin{document}

\title[Schr\"odinger equation with nonlinear subcritical dissipation]{Asymptotic behavior for a Schr\"odinger equation with nonlinear subcritical dissipation}

\author[T. Cazenave]{Thierry Cazenave$^1$}
\email{\href{mailto:thierry.cazenave@sorbonne-universite.fr}{thierry.cazenave@sorbonne-universite.fr}}

\author[Z. Han]{Zheng Han$^2$}
\email{\href{mailto:hanzh_0102@hznu.edu.cn}{hanzh\_0102@hznu.edu.cn}}

\address{$^1$Sorbonne Universit\'e \& CNRS, Laboratoire Jacques-Louis Lions,
B.C. 187, 4 place Jussieu, 75252 Paris Cedex 05, France}

\address{$^2$Department of Mathematics, Hangzhou Normal University, Hangzhou, 311121, China}

\thanks{ZH thanks NSFC 11671353,11401153, Zhejiang Provincial Natural Science Foundation of China under Grant No. LY18A010025, and CSC for their financial support; and the Laboratoire Jacques-Louis Lions for its kind hospitality}


\begin{abstract}
 We study the time-asymptotic behavior of solutions of the Schr\"odinger equation with nonlinear dissipation
\begin{equation*} 
 \partial _t u = i \Delta u + \lambda |u|^\alpha u
\end{equation*} 
in $\R^N $, $N\geq1$, where $\lambda\in\C$, $\Re \lambda <0$ and $0<\alpha<\frac2N$. We give a precise description of the behavior of the solutions (including decay rates in $L^2$ and $L^\infty $, and asymptotic profile), for a class of arbitrarily large initial data, under the additional assumption that $\alpha $ is sufficiently close to  $\frac2N$.
\end{abstract}

\maketitle

\section{Introduction}

In this paper, we consider the  Schr\"odinger equation with nonlinear dissipation
\begin{equation} \label{NLS-0}
\begin{cases} 
 \partial _t u = i \Delta u + \lambda |u|^\alpha u,\\
       u(0,x)=u_0,
\end{cases} 
\end{equation}
where  $\lambda\in\C$ with 
\begin{equation} \label{fAssLam} 
\Re\lambda<0
\end{equation} 
and $0<\alpha<\frac2N$.

Equation~\eqref{NLS-0} is itself a particular case of the more general complex Ginzburg-Landau equation on $\R^N$:
$ u_t=e^{i\theta}\Delta u+ z |u|^\alpha u $, 
where $- \frac {\pi } {2} \leq\theta\leq\frac\pi2$, $z\in \C$ and $\alpha>0$, which is a generic modulation equation describing the nonlinear evolution of patterns at near-critical conditions (see e.g.~\cite{SteStu,CrHo, Mielke}). 

Equation~\eqref{NLS-0} is mass-subcritical, and is globally well-posed in $L^2 (\R^N ) $ and $H^1 (\R^N ) $. See Proposotion~\ref{eRemLWP1} below.

Concerning the large-time asymptotic behavior of the solutions of~\eqref{NLS-0} under assumption~\eqref{fAssLam},  $\alpha=\frac2N$ is a limiting case. Indeed, if $\alpha>\frac2N$, $\lambda\in \C$, then a large set of initial values produces solutions that scatter as $t\to \infty $, i.e. that are asymptotic to a solution of the free Schr\"odinger equation. (See~\cite{Strauss2, GV2,GV1,CW,GOV, NaOz, CCDW, CN1}.) 

If $\alpha\leq\frac2N$, then in many cases solutions are known to decay faster than the solutions of the free Schr\"odinger equation. If $\alpha = \frac {2} {N}$, then for a large class of initial values, the solutions of~\eqref{NLS-0}    can be described by an asymptotic formula, and have the decay rate  $(t\log t)^{-\frac N2}$. 
See~\cite{Shi, KiSh, CN2}. 
In addition, for some solutions,
\begin{equation*} 
(t\log t)^{ \frac N2}  \| u(t) \| _{ L^\infty  } \goto  _{ t\to \infty  } (\alpha  | \Re \lambda |)^{ - \frac {N} {2}}.
\end{equation*} 
See~\cite{CN2}. 

In the one-dimensional case $N=1$, if $\alpha< 2$ is sufficiently close to $2$ and
\begin{equation}\label{res-lam}
\frac{\alpha}{2\sqrt{\alpha+1}}|\Im \lambda|\leq|\Re \lambda|,
\end{equation}
then the large-time asymptotic behavior of solutions can be described for any initial data in $H^1 (\R ) \cap L^2 (\R,  |x|^2 dx )  $, and the  solutions satisfy
\begin{equation}\label{uppbd-u}
\|u(t)\|_{L^\infty}\lesssim t^{-\frac1\alpha},
\end{equation}
see~\cite{KiSh}. 
In addition, in any space dimension $N\ge 1$, under assumption~\eqref{res-lam} and for $\alpha <\frac {2} {N}$ sufficiently close to $\frac {2} {N}$, all solutions with initial value in $H^1 (\R^N ) \cap L^2 (\R ^N ,  |x|^2 dx )  $ satisfy $ \| u(t)\| _{ L^2 } \lesssim t^{-(\frac {1} {\alpha } - \frac {N} {2})q }$ for all $q < \frac {2} {N+2}$, $q\le \frac {1} {2}$. See~\cite{HLN1}. 

In space dimensions $N=1,2,3$ without the condition~\eqref{res-lam},  and for $\alpha <\frac {2} {N}$ sufficiently close to $\frac {2} {N}$, the upper estimate~\eqref{uppbd-u}, as well as lower estimates, is established for sufficiently small initial data in a certain space. See~\cite{HLN2}.

Our purpose in this article is to complete the previous results for~\eqref{NLS-0}, and describe the large-time asymptotic behavior of the solutions for a class of arbitrarily large initial data. In order to state our result, we recall the definition of the space $\Spa $ introduced~\cite{CN1}, which we use in a essential way.
We consider three integers $\Imqt ,\Imqu , \Imqd $ such that
\begin{equation} \label{def-knm} 
\Imqt > \frac {N} {2}, \quad \Imqd > \max \Bigl\{ \frac {N} {2} +1,  \frac {N (N+1)} {4 } \Bigr\}, \quad  2 \Imqu \ge \Imqt + \Imqd +1
\end{equation} 
and we let
\begin{equation} \label{def-J} 
\Imqh = 2\Imqu +2 + \Imqt+ \Imqd .
\end{equation} 
We define the space $\Spa$ by
\begin{equation} \label{fSpa1} 
\begin{split} 
\Spa=  \{ u\in H^\Imqh  (\R^N ); & \, \langle x\rangle ^\Imqd D^\beta u \in L^\infty  (\R^N )  \text{ for  }  0\le  |\beta |\le 2\Imqu , \\   \langle x\rangle ^\Imqd D^\beta u  \in L^2 & (\R^N )  \text{ for  }   2\Imqu +1 \le  |\beta | \le 2\Imqu +2 + \Imqt,  \\  \langle x\rangle ^{\Imqh -  |\beta |} D^\beta u & \in L^2  (\R^N )  \text{ for  }    2\Imqu +2 + \Imqt <  |\beta | \le J \}
\end{split} 
\end{equation} 
and we equip $\Spa$ with the norm
\begin{equation} \label{fSpa2}
 \| u \|_\Spa =  \sum_{ \Imqc =0 }^{2\Imqu  }   \sup _{  |\beta |=    \Imqc }  \| \langle x\rangle ^\Imqd D^\beta  u  \| _{ L^\infty  } +   \sum_{\Imqp =0 }^{\Imqt +1} \sum_{ \Imqs =0 }^\Imqd \sum_{  |\beta  |=\Imqp + \Imqs  +  2 \Imqu  +1  }   \| \langle x \rangle ^{\Imqd -\Imqs } D^\beta  u \| _{ L^2 } 
\end{equation} 
where
\[
\la x\ra=(1+|x|^2)^\frac12 .
\]
In particular,  $(\Spa,  \| \cdot  \|_\Spa )$ is a Banach space and $\Spa \hookrightarrow H^J (\R^N ) $.

Our main result is the following.

\begin{thm}\label{ethm1}
Let $\lambda\in\C$ satisfy~\eqref{fAssLam}, assume~\eqref{def-knm}-\eqref{def-J} and let $ \Spa $ be defined by~\eqref{fSpa1}-\eqref{fSpa2}. 
Given any $K>1$, there exist $\frac {2} {N+1}< \alpha _1< \frac {2} {N}$ and $b_1>0$ with the following property. 
Let $\alpha _1 \le \alpha < \frac {2} {N}$. 
Suppose $u_0(x)=e^{i\frac{b|x|^2}{4}}v_0(x)$, where $b \ge  b_1$ and $v_0\in \Spa $ satisfies
\begin{equation}\label{ID}
\inf\limits_{x\in\R^N}\la x\ra^n|v_0(x)|>0.
\end{equation}
and
\begin{equation}\label{IV}
\|v_0\|_{ \Spa }+\Big(  \inf\limits_{x\in\R^N}\la x\ra^n|v_0(x)|\Big)^{-1}\leq K. 
\end{equation}
It follows that the corresponding solution $u\in C([0,\infty),H^1 (\R^N ) )$ of~\eqref{NLS-0} belongs to $L^\infty  ((0, \infty ) \times \R^N  ) $. Moreover,  there exist $f_0, \omega_0\in L^\infty (\R^N )$, with $f_0$ real-valued and $\|f_0\|_{L^\infty}\leq\frac12$,  and $\la \cdot \ra^n\omega_0\in L^\infty(\R^N)$, such that
\begin{equation}\label{asy-u}
t ^{\frac {1} {\alpha } - \frac {N} {2}}\|u( t )-z(t )\|_{L^2}+ t^{\frac 1 \alpha }\|u(t )-z(t )\|_{L^\infty}\leq C t ^{-( \frac {N+1} {2} - \frac {1} {\alpha })},
\end{equation}
where
\begin{equation}
z(t,x)=(1+bt)^{-\frac{N}{2}}e^{i\Theta(t,x)}\Psi  \Bigl(  t,\frac{x}{1+bt}  \Bigr) \omega_0  \Bigl(  \frac{x}{1+bt}  \Bigr)
\end{equation}
with
\begin{equation}
\Theta(t,x)=\frac{b|x|^2}{4(1+bt)}- \frac{ \Im \lambda}{ \Re \lambda}\log  \Bigl(  \Psi  \Bigl(  t,\frac{x}{1+bt}  \Bigr)  \Bigr)
\end{equation}
and
\begin{equation}\label{def-Psi}
\Psi(t,y)=\Big( \frac{1+f_0(y)}{1+f_0(y)+\frac{ 2 \alpha | \Re \lambda|}{b ( 2-N\alpha ) }|v_0(y)|^\alpha[(1+bt)^{\frac{2-N\alpha}{2}}-1]}\Big)^\frac1\alpha.
\end{equation}
Moreover,
\begin{equation} \label{fDscromzero-u} 
 |\omega _0 |^\alpha = \frac { |v_0 |^\alpha } {1+ f_0},
\end{equation} 
so that $\frac {2} {3}  | v_0|^\alpha  \le  |\omega _0|^\alpha  \le 2 | v_0|^\alpha $.
Furthermore,
\begin{equation}\label{decay-u}
t\|u(t)\|_{L^\infty}^\alpha  \goto_{ t\uparrow\infty }\frac{2-N\alpha}{2\alpha| \Re \lambda|},
\end{equation}
and 
\begin{equation}\label{decay-l2-u}
t^{- (\frac {1} {\alpha } - \frac {N} {2}) ( 1- \frac {N} {2n} )} \lesssim \|u(t)\|_{L^2} \lesssim t^{- (\frac {1} {\alpha } - \frac {N} {2}) ( 1- \frac {N} {2n} )}
\end{equation}
as $t\to \infty $, where $n$ is given by~\eqref{def-knm}.  
\end{thm}

\begin{rem} \label{eRem2} 
Here are some comments on Theorem~\ref{ethm1}.
\begin{enumerate}[{\rm (i)}] 

\item \label{eRem2:2} 
We have $v_0\in \Spa \hookrightarrow L^2 (\R^N ) $,  $\Spa \hookrightarrow H^1 (\R^N ) $, and $\Spa \hookrightarrow L^2 (\R^N ,  |x|^2 dx) $ (because $n> \frac {N} {2} +1$), so that $u_0\in H^1 (\R^N ) $. Therefore,  the solution $u\in C([0,\infty), H^1 (\R^N ) )$ of~\eqref{NLS-0} is well defined, see Proposotion~\ref{eRemLWP1}.
Moreover, $u$ is smoother than stated. Indeed, $u$ is given by the pseudo-conformal transformation~\eqref{fPseudoConf2} in terms of a solution $v\in C([0, \frac {1} {b}), \Spa )$  of equation~\eqref{NLS-1}. In particular,  $u$ is a classical solution of~\eqref{NLS-0} ($C^1$ in $t$ and $C^2$ in $x$).

\item \label{eRem2:3} 
Theorem~\ref{ethm1} is valid in any dimension $N\geq1$ and for any $\lambda\in\C$ with $\Re \lambda<0$. In particular, we do not require assumption~\eqref{res-lam}. The main restrictions are that $\alpha $ must be sufficiently close to $\frac {2} {N}$ and that the initial value must be bounded from below in the sense~\eqref{ID} and sufficiently oscillatory in the sense that $b$ must be sufficiently large. Moreover, how close $\alpha $ must be to $\frac {2} {N}$ depends on a certain bound on the initial value through~\eqref{IV}. 
On the other hand, there is no restriction on the size of $u_0$.

\item \label{eRem2:1} 
A typical initial value which is admissible in Theorem~\ref{ethm1} is  $\DIb = c  \langle \cdot \rangle ^{-n} + \varphi $ with $c\in \C$, $c\not = 0$, and $\varphi \in \Srn$, $ |\varphi | \le ( |c| -\varepsilon )  \langle \cdot \rangle ^{-n} $, $\varepsilon >0$.
Indeed, it is easy to check that $\DIb \in \Spa$ and $\DIb$ satisfies~\eqref{ID}. 
Then $K$ must be chosen sufficiently large so that~\eqref{IV} holds and $\alpha $ sufficiently close to $\frac {2} {N}$.
Note that any value of $n$ sufficiently large so that the second condition in~\eqref{def-knm} is satisfied, is admissible. 

\item \label{eRem2:4} 
The limit~\eqref{decay-u} gives the exact decay rate of $ \| u(t)\| _{ L^\infty  }$. Note that this limit is independent of the initial value $u_0$. The reason for this is that~\eqref{decay-u}  is equivalent to~\eqref{converge}, and that the solutions of the ODE $z'= \lambda (1-bt)^{-\frac {4- N\alpha } {2}}  |z|^\alpha z$ satisfy $(1-bt)^{-\frac{2-N\alpha}{2}}  |z(t)|^\alpha \to \frac{b (2-N\alpha ) }{ 2 \alpha| \Re \lambda|}$ as $t\uparrow \frac {1} {b}$, independently of the initial value $z(0)$. 

\item 
It follows from~\eqref{decay-l2-u} that $ \|u(t)\|_{L^2} $ is equivalent to $ t^{- (\frac {1} {\alpha } - \frac {N} {2}) ( 1- \frac {N} {2n} )} $. 
With respect to the results in~\cite{HLN1}, \eqref{decay-l2-u} gives the exact decay rate of $ \|u(t)\|_{L^2} $.
As opposed to the decay rate of $ \| u(t)\| _{ L^\infty  }$, which is $t^{-\frac {1} {\alpha }}$ (hence independent of the solution), the decay rate of $ \|u(t)\|_{L^2} $ does depend on the solution, through the parameter $n$ which can be chosen (as long as it is sufficiently large).

\item \label{eRem2:5} 
It follows from~\eqref{decay-u} and~\eqref{decay-l2-u} that 
\begin{equation*} 
\liminf _{ t\to \infty  }  t^{\frac 1 \alpha }\|u(t ) \|_{L^\infty} >0 , \quad   t ^{\frac {1} {\alpha } - \frac {N} {2}}\|u( t ) \|_{L^2} \goto _{ t\to \infty  } \infty  .
\end{equation*} 
Thus we see  that the asymptotic behavior of $u (t)$ as $t\to \infty $ is described by the function $z(t)$ via the estimate~\eqref{asy-u}.
Note that the functions $f_0$ and $\Psi $ are both real-valued, and that $\frac {1} {2} \le 1+ f_0 \le \frac {3} {2}$
 and $0< \Psi \le 1$. 
The function $\Theta $ is also real-valued. If $\Im \lambda \le 0$, then $\Theta >0$. If $\Im \lambda >0$, then $\Theta $ takes both positive and negative values. 

\end{enumerate} 
\end{rem} 

\begin{rem}
If $\Re \lambda >0$, then finite-time blowup occurs for equation~\eqref{NLS-0}, at least for $H^1$-subcritical powers $(N-2) \alpha < 4$.  See~\cite{CMZ,CMHZ}. Moreover, if $\alpha < \frac {2} {N}$, then all nontrivial solutions blow up in finite or infinite time, see~\cite{CCDW}. Finite-time blowup also occurs if $\Re \lambda =0$, $\Im \lambda >0$, and $\alpha \ge  \frac {4} {N}$, since in this case~\eqref{NLS-0}  is the focusing NLS. 
If $\Re \lambda <0$, $\alpha > \frac {4} {N}$ and condition~\eqref{res-lam} is not satisfied, then whether or not some solutions of~\eqref{NLS-0} blow up in finite time seems to be an open question. 
\end{rem}

We apply the strategy of~\cite{CN1,CN2} to prove Theorem~\ref{ethm1}. We require the non-vanishing condition~\eqref{ID}, as well as strong decay and regularity of the the initial data to overcome the difficulty of non-smooth nonlinearity and derivative loss in their estimates. This is why the various conditions in the definition of the space $ \Spa $ arise. The other main ingredient is the application of the pseudo-conformal transformation. Given any $b>0$, by the pseudo-conformal transformation
\begin{equation} \label{fForPseuCon} 
v(t,x)= (1-bt) ^{ - \frac {N} {2}} u  \Bigl( \frac {t} {1-bt} , \frac {x} {1-bt} \Bigr) e^{- i \frac { b|x|^2} {4(1-bt)}}\quad t\geq0,\ x\in\R^N,
\end{equation}
equation~\eqref{NLS-0} is equivalent to the nonautonomous equation
\begin{equation}\label{NLS-1}
\begin{cases} 
\partial _t v = i\Delta v+\lambda(1-bt)^{-\frac{4-N\alpha}{2}}|v|^\alpha v,\\
v(0,x)=v_0.
\end{cases} 
\end{equation}
Note that the assumption $\alpha\leq\frac2N$ implies that $(1-bt)^{-\frac{4-N\alpha}{2}}$ is not integrable at $1/b$. As in~\cite{CN2}, we estimate the solution $v(t,x)$ by allowing a certain growth of the various components of the $ \Spa $-norm of the solution,  see~\eqref{def-P1}-\eqref{def-P4}.  Using Duhamel's formula for~\eqref{NLS-1}, i.e.
\begin{equation}\label{Duham-form}
v(t)=e^{it\Delta}v_0 + \lambda\int^t_0(1-bs)^{-\frac{4-N\alpha}{2}}e^{i(t-s)\Delta}|v(s)|^\alpha v(s)\,ds
\end{equation}
and the elementary calculation
\begin{equation} \label{fEIn1} 
\int^t_0(1-bs)^{-1- \nu }\,ds=\frac1{b \nu }[(1-bt)^{- \nu }-1]\leq \frac1{b \nu }(1-bt)^{- \nu },
\end{equation}
we see that if $e^{i(t-s)\Delta}|v(s)|^\alpha v(s)$ is estimated in a certain norm by $(1-bs)^{-\mu}$, then $v(t)$ can be controlled in that norm by $(1-bs)^{-\mu+\frac{2-N\alpha}{2}}$. 
In the case $\alpha = \frac2N$, one obtains the same power $(1-bs)^{-\mu}$, and this can be used to close appropriate estimates. This is the strategy employed in~\cite{CN2}. 
In the present case $\alpha < \frac {2} {N}$, we observe that if $e^{i(t-s)\Delta}|v(s)|^\alpha v(s)$ is estimated in a certain norm by $(1-bs)^{-\mu+\frac{2-N\alpha}{2}}$, then $v(t)$ can be controlled in that norm by $(1-bs)^{-\mu }$. 
We obtain the extra decay by monitoring the decay of  $|v(s)|$ (see Lemma~\ref{lem-Decay-v}). The price to be paid is that the constants that appear in the calculations not only depend on $1/b$, but also on  $\frac{2-N\alpha}{| \Re \lambda|}$. 
Therefore, in order to close the estimates, we are led to require not only that $b$ is large, but also that  $\alpha$ is close to $\frac2N$.

The rest of the paper is organized as follows.  in section~\ref{sPrelim}, we recall some estimates and a local well-posedness result in the space $\Spa$ for equation~\eqref{NLS-1}, taken from~\cite{CN1,CN2}. The crucial estimate of the solutions is carried out in Section~\ref{sEstimates}. Using these estimates, we describe in Section~\ref{sAsympt} the asymptotic behavior of the corresponding solutions of~\eqref{NLS-1}. Finally, we complete the proof of Theorem~\ref{ethm1} in section~\ref{sPrfThm1}, by applying the pseudo-confirmal transformation.

\section{Preliminary} \label{sPrelim} 

We recall some properties of equation~\eqref{NLS-0} which will be useful in the next sections. 
We begin with a global well-posedness result.

\begin{prop} \label{eRemLWP1} 
Let $0<\alpha < \frac {4} {N}$ and let $\lambda \in \C$ satisfy $\Re \lambda \le 0$. It follows that the Cauchy problem~\eqref{NLS-0} is  globally well-posed in $L^2 (\R^N ) $ and in $H^1 (\R^N ) $. 
More precisely, given any $u_0 \in L^2 (\R^N ) $ there exists a solution $u\in C( [0, \infty ), L^2 (\R^N ) ) \cap L^{\alpha +2}  ( [0, \infty ), L^{\alpha +2} (\R^N ) )$ of~\eqref{NLS-0}. 
The solution is unique and depends continuously on $u_0$ in $C ( [0, T ), L^2 (\R^N ) ) \cap L^{\alpha +2}  ((0, T ), L^{\alpha +2} (\R^N ) )$ for every $T>0$. If, in addition, $u_0\in H^1 (\R^N ) $, then $u\in C([0,\infty ), H^1 (\R^N ) )$. 
\end{prop} 

\begin{proof} 
For the local theory (local existence, uniqueness, continuous dependence, regularity), see  e.g.~\cite{Kato1, Kato2}. For global existence, it is sufficient to estimate the $L^2$ norm of $u$.
Multiplying~\eqref{NLS-0} by $ \overline{u} $, taking the real part and integrating by parts, we obtain
\begin{equation} \label{fIDEC1} 
 \| u(t) \| _{ L^2 }^2 + (- \Re \lambda ) \int _0^t  \| u(s) \| _{ L^{\alpha +2} }^{\alpha +2} =  \| u_0\| _{ L^2 }^2 .
\end{equation} 
(This argument is formal, but~\eqref{fIDEC1} can be proved by standard approximation arguments, see for instance~\cite{Ozawa2}.)
It follows that $u$ is bounded in $L^2 (\R^N ) $. 
\end{proof} 

Next, we recall some estimates for the Schr\"o\-din\-ger  equation in the space $ \Spa $.

\begin{prop}[\cite{CN2}, Propositon~2.1]\label{est-solu}
Assume~\eqref{def-knm}-\eqref{def-J} and let $ \Spa $ be defined by~\eqref{fSpa1}-\eqref{fSpa2}. 
There exists $\CST1\geq1$ such that if $T\geq0,v_0\in \Spa $ and $f\in C([0,T], \Spa )$, then for all $0\leq t\leq T$, the solution $v$ of
\begin{equation*} 
\begin{cases} 
 \partial _t v= i \Delta v + f,\\ v(0,x)=v_0,
\end{cases} 
\end{equation*} 
satisfies the following estimates.
\begin{enumerate}[{\rm (i)}] 
 \item If $|\beta|\leq2m$, then
\begin{equation}\label{est-solu-small-der}
\|\la x\ra^nD^\beta v(t)\|_{L^\infty}\leq\|v_0\|_{ \Spa }+ \CST1\int^t_0(\|v(s)\|_{ \Spa }+\|\la x\ra^nD^\beta f(s)\|_{L^\infty})\,ds.
\end{equation}
  \item If $|\beta|=\nu+\mu+2m+1\ with\ 0\leq\nu\leq k+1\ and\ 0\leq\mu\leq n$, then
\begin{equation}\label{est-solu-lar-der}
\|\la x\ra^{n-\mu}D^\beta v(t)\|_{L^2}\leq\|v_0\|_{ \Spa }+ \CST1\int^t_0(\|v(s)\|_{ \Spa }+\|\la x\ra^{n-\mu}D^\beta f(s)\|_{L^2})\,ds.
\end{equation}
\end{enumerate} 
\end{prop}

We now recall several estimates of the nonlinearity $|v|^\alpha v$. Given $\ell \in \N$, we set
\begin{equation} \label{fCA2} 
\|  u\| _{1, \ell} = \sup   _{ 0\le  |\beta | \le \ell  } \| \langle \cdot \rangle^{n}D^{\beta}u\| _{L^{\infty}}
\end{equation} 
\begin{equation} \label{fCA3}
\|  u\| _{2, \ell} =
\begin{cases} 
\displaystyle  \sup_{ 2{m}  +1\le  |\beta | \le \ell } \| \langle \cdot \rangle^{n}D^{\beta }u\| _{L^2  } &  \ell \ge 2m+1 \\
0 & \ell \le 2m 
\end{cases} 
\end{equation} 
and
\begin{equation} \label{fCA4} 
\|  u\| _{3, \ell } = 
\begin{cases} 
\displaystyle  \sup_{ 2{m}+3+{k} \le  |\beta | \le \ell }\| \langle \cdot \rangle^{{J}-\ell }D^{\beta }u\| _{L^ 2 } & \ell \ge 2m +3 + k \\ 0 & \ell \le 2m + 2 + k .
\end{cases} 
\end{equation} 
We have the following estimates of the nonlinearity.

\begin{prop}[\cite{CN2}, Proposition~3.1]\label{est-nonlin}
Assume~\eqref{def-knm}-\eqref{def-J} and let $ \Spa $ be defined by~\eqref{fSpa1}-\eqref{fSpa2}. 
Let $\alpha >0$ and suppose that, in addition to~\eqref{def-knm}, $n\ge \frac {N} {2\alpha }$. 
It follows that there exists a constant $ \CST2\geq1$ such that if $v\in  \Spa $ satisfies
\begin{equation} \label{fNDC1} 
\eta\inf\limits_{x\in\R^N}(\la x\ra^n|v(x)|)\geq1,
\end{equation}
for some $\eta >0$, then the following estimates hold.
\begin{enumerate}[{\rm (i)}] 

 \item 
 If $0\leq|\beta|\leq1$, then
\begin{equation}\label{nonlin-est-1}
\|\la \cdot \ra^nD^\beta ( |v|^\alpha v)  \|_{L^\infty}\leq  \CST2\|v\|_{L^\infty}^\alpha  \|\la \cdot \ra^nD^\beta v\|_{L^\infty} .
\end{equation}

\item 
If $2\leq|\beta|\leq2m$, then
\begin{equation} \label{nonlin-est-2}
\begin{split} 
\|  \langle \cdot \rangle ^n  D^{\beta} ( |v|^\alpha v)   \|  _{L^{\infty}}  \le &  \CST2 \| v \| _{ L^\infty  }^\alpha  \| \langle \cdot\rangle ^{n}D^{\beta} v \| _{L^\infty} \\ &    +  \CST2  \| v \| _{ L^\infty  }^\alpha  ( \eta \|  v\| _{1, | \beta |  -1})^{2|  \beta|  }\|  v\|  _{1,  |  \beta|  -1} .
\end{split} 
\end{equation} 

\item 
If $2m+1\leq|\beta|\leq2m+2+k$, then
\begin{equation} \label{nonlin-est-3}
\begin{split} 
\|\la \cdot \ra^nD^\beta(|v|^\alpha v)\|_{L^2} \leq & \CST2\|v\|_{L^\infty}^\alpha\|\la \cdot \ra^nD^\beta v\|_{L^2}\\ &
+\CST2(\eta\|v\|_{1,2m})^{2J+\alpha}\|v\|_{1,2m} \\ &+ \CST2\|v\|_{L^\infty}^\alpha(\eta\|v\|_{1,2m})^{2J}\|v\|_{2,|\beta|-1}.
\end{split} 
\end{equation} 
  
 \item 
 If $2m+3+k\leq|\beta|\leq J$, then
\begin{equation} \label{nonlin-est-4}
\begin{split} 
\|\la \cdot \ra^{J-|\beta|}D^\beta(|v|^\alpha v)\|_{L^2}\leq &  \CST2\|v\|_{L^\infty}^\alpha\|\la \cdot \ra^{J-|\beta|}D^\beta v\|_{L^2}
\\ & + \CST2(\eta\|v\|_{1,2m})^{2J+\alpha}\|v\|_{1,2m}\\
+ \CST2\|v\|_{L^\infty}^\alpha & (\eta\|v\|_{1,2m})^{2J}(\|v\|_{2,2m+2+k}+\|v\|_{3,|\beta|-1}).
\end{split} 
\end{equation} 

\end{enumerate} 
\end{prop}

\begin{rem} 
Estimates~\eqref{nonlin-est-2}--\eqref{nonlin-est-4} are not exactly the estimates of~\cite[Proposition~3.1]{CN2}. 
First, $1+ \eta \|  v\| _{1, \ell}$ is replaced by $\eta \|  v\| _{1, \ell}$ (with $\ell = |\beta |-1$ in~\eqref{nonlin-est-2} and $\ell = 2m$ in~\eqref{nonlin-est-3} and~\eqref{nonlin-est-4}). The two quantities are indeed equivalent, since by~\eqref{fNDC1}, $ \eta \|  v\| _{1, \ell} \ge 1$. 
Next, the term $(\eta\|v\|_{1,2m})^{2J+\alpha}\|v\|_{2,  |\beta | -1} $ in~\cite[formula~(3.9)]{CN2} is replaced in formula~\eqref{nonlin-est-3} here by $  \| v \| _{ L^\infty  }^\alpha  (\eta\|v\|_{1,2m})^{2J }\|v\|_{2,  |\beta | -1} $. 
This is in fact what the proof in~\cite{CN2} shows, see in particular~\cite[formula~(3.20)]{CN2}.
Finally, the term $ (\eta\|v\|_{1,2m})^{2J+ \alpha }(\|v\|_{2,2m+2+k}+\|v\|_{3,|\beta|-1})$ in~\cite[formula~(3.10)]{CN2} is replaced in formula~\eqref{nonlin-est-4} here by $\|v\|_{L^\infty}^\alpha (\eta\|v\|_{1,2m})^{2J}(\|v\|_{2,2m+2+k}+\|v\|_{3,|\beta|-1})$. Again, this is what the proof in~\cite{CN2} shows, see in particular~\cite[formulas~(3.24) and~(3.25)]{CN2}.
The term $  \| v \| _{ L^\infty  }^\alpha $ in these estimates is important in our proof of Proposition~\ref{BD-GWP} below.

\end{rem} 

Finally, we recall the local well-posedness of ~\eqref{NLS-1} in the space $\Spa$, see~\cite[Theorem~1]{CN1} and~\cite[Proposition~4.1]{CN2}.

\begin{prop}\label{LWP}
Assume~\eqref{def-knm}-\eqref{def-J} and let $ \Spa $ be defined by~\eqref{fSpa1}-\eqref{fSpa2}. 
Let $\alpha >0$ and suppose that, in addition to~\eqref{def-knm}, $n\ge \frac {N} {2\alpha }$. 
Let $\lambda\in\C$ and $b\geq0$. If $v_0\in \Spa $ satisfies
\begin{equation}\label{inf-0}
\inf\limits_{x\in\R^N}\la x\ra^n|v_0(x)|>0,
\end{equation}
then there exist $0<T<\frac1b$ and a unique solution $v\in C([0,T], \Spa )$ of ~\eqref{NLS-1} satisfying
\begin{equation}\label{inf}
\inf\limits_{0\leq t\leq T}\inf\limits_{x\in\R^N}(\la x\ra^n|v(t,x)|)>0.
\end{equation}
Moreover, $v$ can be extended on a maximal existence interval $[0, \Tma )$ with $0< \Tma \leq\frac1b$ to a solution $v\in C([0, \Tma ), \Spa )$ satisfying~\eqref{inf} for all $0<T< \Tma $; and if $ \Tma <\frac1b$, then
\begin{equation}\label{blowup}
\|v(t)\|_{ \Spa }+\Big( \inf\limits_{x\in\R^N}\la x\ra^n|v(t,x)|\Big)^{-1} \goto_{ t\uparrow  \Tma }\infty.
\end{equation}
\end{prop}

\section{Estimates for~\eqref{NLS-1}} \label{sEstimates} 
Throughout this section, we assume~\eqref{def-knm}-\eqref{def-J} and we let $ \Spa $ be defined by~\eqref{fSpa1}-\eqref{fSpa2}. 
We derive estimates for certain solutions of~\eqref{NLS-1}. 
We first introduce several indices and seminorms.
Let 
\begin{align} 
\sigma _0 &= 0,  \label{fDfnsigma0} \\
\sigma _1 &= \frac {1 } { 4 [4J (J-2m-1) + 4J+ (4/N) +1] (8m+1)^{2m}  },  \label{fDfnsigma1}
\end{align} 
and set
\begin{equation}\label{de-sigma}
\sigma_j=\begin{cases}
          (8m+1)^j\sigma_1,&2\leq j\leq2m\\
          \frac{2-N\alpha}{2}+(4J+2\alpha+1)\sigma_{2m},&\ j=2m+1\\
          4J\sigma_{2m}(j-2m-1)+\sigma_{2m+1},&2m+2\leq j\leq J.
         \end{cases}
\end{equation}
It follows that
\begin{equation}
0=\sigma_0<\sigma_1\leq\sigma_j<\sigma_k\leq\sigma_J , \quad 1\leq j<k\leq J.
\end{equation}
Moreover, it follows from~\eqref{de-sigma} that
\begin{equation} \label{festsigj1} 
\sigma _J = [4J (J-2m-1) + 4J+2\alpha +1] (8m+1)^{2m} \sigma _1 + \frac {2-N\alpha } {2}. 
\end{equation} 
We deduce from~\eqref{festsigj1} and~\eqref{fDfnsigma1} that
\begin{equation} \label{festsigj2} 
\sigma _J \le  \frac {1} {2} \quad  \text{for } \alpha \in [ \frac {3 } {2 N}, \frac {2} {N}].
\end{equation} 
Given $0<T<\frac1b$ and $v\in C([0,T], \Spa )$ satisfying~\eqref{inf}, we define
\begin{align}
\Phi_{1,T}&=\sup\limits_{0\leq t\leq T}\sup\limits_{0\leq j\leq 2m}(1-bt)^{\sigma_j}\|v\|_{1,j}\label{def-P1}\\
\Phi_{2,T}&=\sup\limits_{0\leq t\leq T}\sup\limits_{2m+1\leq j\leq 2m+2+k}(1-bt)^{\sigma_j}\|v\|_{2,j}\label{def-P2}\\
\Phi_{3,T}&=\sup\limits_{0\leq t\leq T}\sup\limits_{2m+3+k\leq j\leq J}(1-bt)^{\sigma_j}\|v\|_{3,j}\label{def-P3}\\
\Phi_{4,T}&=\sup\limits_{0\leq t\leq T}\frac{(1-bt)^{\sigma_1}}{\inf\limits_{x\in\R^N}\la x\ra^n|v(t,x)|}\label{def-P4}
\end{align}
where the norms $\|\cdot\|_{l,j}$ are given by ~\eqref{fCA2}--\eqref{fCA4}, and we denote
\begin{align}
\Phi_T&=\max\{\Phi_{1,T},\Phi_{2,T},\Phi_{3,T}\}\\
\Psi_T&=\max\{\Phi_T,\Phi_{4,T}\}.
\end{align}
From these definitions, it is easy to verify that
\begin{align}
\|v\|_{L^\infty((0,t), \Spa )}&\leq  \CST3(1-bt)^{-\sigma_J}\Psi_T,\label{uppbd-1}\\
\|v(t)\|_{ \ell ,j}&\leq(1-bt)^{-\sigma_j}\Psi_T,\quad \ell= 1,2,3\label{uppbd-2}\\
\frac{1}{\la x\ra^n|v(t,x)|}&\leq(1-bt)^{-\sigma_1}\Psi_T,\label{uppbd-3}
\end{align}
where the constant $ \CST3\geq1$ is independent of $t\in [0,T]$.

\begin{lem}\label{lem-Decay-v}
Suppose $\Re \lambda<0$ and $ \frac {3 } {2 N} \le \alpha \le \frac {2} {N}$. 
Let $K\ge 1$ and set 
\begin{equation} \label{bd-f-1} 
b_0 = \frac {16} {N} (4K)^{ \frac {4} {N} +2} .
\end{equation}
Let $b>0$, let $v_0\in \Spa $ satisfy~\eqref{inf-0}, and let $v\in C([0, \Tma), \Spa )$ be the solution of ~\eqref{NLS-1} given by Proposition~$\ref{LWP}$. If $v$ satisfies
\begin{equation}\label{upbd-v}
\sup _{0\leq t\leq T}\Psi_{t}\leq 4K,
\end{equation}
for some $0< T < \Tma$ and if $b\ge b_0$, then
\begin{equation}\label{Decay-v}
\|v(t )\|_{L^\infty}^\alpha\leq \frac{b(2-N\alpha)}{2\alpha| \Re \lambda|}\frac{(1-bt)^\frac{2-N\alpha}{2}}{1-(1-bt)^\frac{2-N\alpha}{2}},
\end{equation}
for all  $0< t\le T$.
\end{lem}

\begin{proof}
Multiplying~\eqref{NLS-1} by $  \overline{v} $, taking the real part and using that $|v|>0$ on $[0, \Tma ) \times\R^N$ by Proposition~\ref{LWP}, we obtain
\begin{equation}\label{eq-v-1}
\partial _t |v| =L+ \Re \lambda(1-bt)^{-\frac{4-N\alpha}{2}}|v|^{\alpha+1}
\end{equation}
where 
\begin{equation} \label{fDfnL} 
L(t,x)=-\frac{\Im(  \overline{v} \Delta v)}{|v|}; 
\end{equation} 
and so
\begin{equation}\label{eq-v-2}
-\frac1\alpha  \partial _t ( |v|^{-\alpha} )=|v|^{-\alpha-1}L+ \Re \lambda(1-bt)^{-\frac{4-N\alpha}{2}}.
\end{equation}
Integrating~\eqref{eq-v-2} in $t$, we obtain
\begin{equation*}
\begin{split} 
\frac{1}{|v(t,x)|^\alpha}= & \frac{1}{|v_0(x)|^\alpha}+\frac{ 2 \alpha| \Re \lambda |}{( 2-N\alpha ) b} [ (1-bt)^{-\frac{2-N\alpha}{2}}-1] \\ & -\alpha\int^t_0|v(s,x)|^{-\alpha-1}L(s,x)\,ds
\end{split} 
\end{equation*}
so that
\begin{equation}\label{eq-v-4}
|v(t,x)|^\alpha=\frac{|v_0 (x) |^\alpha} {1+f(t,x)+ \frac{ 2 \alpha| \Re \lambda |}{ ( 2-N\alpha )b} |v_0 (x) |^\alpha[(1-bt)^{-\frac{2-N\alpha}{2}}-1]}
\end{equation}
where 
\begin{equation} \label{fDfnf1} 
f(t,x)=-\alpha\int^t_0|v_0(x)|^\alpha|v(s,x)|^{-\alpha-1}L(s,x)\,ds.
\end{equation} 
It follows from the definitions of $\Psi_T$ and $L(t,x)$ that, for any $0\leq s\leq t\leq T$
\begin{equation} \label{fEstf1} 
\begin{split} 
|v_0(x)|^\alpha|v(s,x)|^{-\alpha-1}|L(s,x)| & \leq(\la x\ra^n|v_0|)^\alpha(\la x\ra^n|v|)^{-\alpha-1}(\la x\ra^n|\Delta v|)\\
& \leq (4K)^{2\alpha+2}(1-bs)^{-(\alpha+1)\sigma_1-\sigma_2} \\ & \leq (4K)^{2\alpha+2}(1-bs)^{- \sigma _3 }
\end{split} 
\end{equation}
where in the last inequality we used  
\begin{equation*} 
(\alpha+1)\sigma_1+\sigma_2 \le (\alpha +2) \sigma _2 \le 4 \sigma _2 \le \sigma_3 
\end{equation*} 
 by~\eqref{de-sigma}. 
Using $\sigma _3\le \sigma _J\le \frac {1} {2}$ (see~\eqref{festsigj2}), we obtain
\begin{equation}\label{bd-f}
|f(t,x)|\leq\alpha\int^t_0 (4K)^{2\alpha+2}(1-bs)^{-\frac {1} {2} }\,ds\leq\frac{ 2 \alpha (4K)^{2\alpha+2}}{b }.
\end{equation}
We deduce from~\eqref{bd-f} and~\eqref{bd-f-1}  that 
\begin{gather}
\frac {3} {2} \ge 1+f(t,x)\geq\frac12,
\end{gather}
for $b\ge b_0$, $0\le t\le T$ and $x\in \R^N $. 
In particular, $1+ f \ge 0$ and estimate~\eqref{Decay-v} follows.
\end{proof}  

\begin{prop}\label{BD-GWP}
Suppose $\Re \lambda<0$.
Given $K\ge 1$, let $\alpha _1 \in ( \max\{ \frac {3} {2N}, \frac {2} {N+1} \}, \frac {2} {N} ) $ be given by
\begin{equation} \label{fDfnalphaun} 
\frac {12 \CST1 \CST2  (4K)^{4J+1} | \lambda|  } {\sigma _1 | \Re \lambda|}   \Bigl(  \frac{ 2 }{   \alpha _1} -N \Bigr) = 1, 
\end{equation} 
and let
\begin{equation} \label{fDfnbun} 
b_1 = \max \Bigl\{  \frac {16} {N} (4K)^{ \frac {4} {N} +2},  8 \CST3 ,  \frac { 32  (4K)^{4J+ 4 } |\lambda | \CST1 \CST2   } {  \sigma  _1 }, \frac { 2^{ \frac {4} {N} +3} \alpha (4K)^{2}} { 3^{\frac {1} {N}} -1} \Bigr\},
\end{equation} 
where $\sigma _1$  is given by~\eqref{fDfnsigma1},  $ \CST1$ by Proposition~$\ref{est-solu}$,  
$ \CST2$ by Proposition~$\ref{est-nonlin}$, and $\CST3$ by~\eqref{uppbd-1}.
If $v_0\in \Spa $ satisfies~\eqref{IV}, 
then for every $\alpha\in[\alpha_1,\frac2N)$ and $b\geq b_1$, the corresponding solution $v\in C([0, \Tma ), \Spa )$ of~\eqref{NLS-1} given by Proposition~$\ref{LWP}$ satisfies $ \Tma =\frac1b$ and
\begin{equation}\label{UB}
\sup\limits_{0\leq T<\frac1b}\Psi_T\leq 4K.
\end{equation}
\end{prop}

\begin{proof}
We set
\begin{equation} \label{fDfnTstar} 
\Tstar =\sup\{0\leq T< \Tma ;\Psi_T\leq4K\}.
\end{equation} 
Since $\Psi_0\leq K$ and $v\in C([0, \Tma ), \Spa )$, we see that $0< \Tstar \leq  \Tma $.
We claim that if $\alpha \in [\alpha _1, \frac {2} {N} )$ and $b\ge b_1$, then
\begin{equation}\label{T*=max}
\Tstar = \Tma .
\end{equation}
We note that, since $\alpha \ge \alpha _1\ge  \frac {2} {N+1}$, the second condition in~\eqref{def-knm} implies that 
$\Imqd > \max \{ \frac {N} {2} +1,  \frac {N} {2\alpha } \}$, so that we may apply Propositions~\ref{est-nonlin} and~\ref{LWP}.
Assuming~\eqref{T*=max},  it follows from~\eqref{uppbd-1}, \eqref{uppbd-3} and~\eqref{fDfnTstar}  that for any $t\in[0, \Tma )$
\begin{equation}\label{upbd-v3}
\|v(t)\|_{ \Spa }+\Big( \inf\limits_{x\in\R^N}\la x\ra^n|v(t,x)|\Big)^{-1}\leq CK(1-bt)^{-\sigma_J}.
\end{equation}
If $ \Tma <\frac1b$, then we deduce from~\eqref{upbd-v3} that
\begin{equation*}
\sup  _{ 0\le t<\Tma  }  \Bigl( \|v(t)\|_{ \Spa }+\Big( \inf\limits_{x\in\R^N}\la x\ra^n|v(t,x)|\Big)^{-1} \Bigr) <\infty ,
\end{equation*}
which contradicts the blowup alternative~\eqref{blowup}. Therefore, we have $ \Tstar  = \Tma =\frac1b$ and~\eqref{UB}  follows.

Now we prove the claim~\eqref{T*=max}. We assume by contradiction that
\begin{equation}
\Tstar < \Tma ,
\end{equation}
then by the definition of $ \Tstar $, we have
\begin{equation}\label{bd-T*}
\Psi_{ \Tstar }=4K.
\end{equation}
It follows from~\eqref{uppbd-1}, \eqref{bd-T*} and~\eqref{festsigj2}  that
\begin{equation}\label{bd-inter-v}
\int^{ \Tstar  }_0 \|v(s)\|_{ \Spa }\,ds\leq 4K \CST3\int^{ \Tstar } _0(1-bs)^{-\sigma_J}\,ds\leq \frac{8K \CST3}{b }.
\end{equation}
Using also~\eqref{IV} and~\eqref{fDfnbun}, we see that
\begin{equation} \label{fEstdpt1} 
 \| v_0 \|_\Spa + \int _0^{\Tstar } \|v(s)\|_{ \Spa }\,ds\leq  2K.
\end{equation} 
Next, we set
\begin{equation}
\eta(t)=4K(1-bt)^{-\sigma_1}\label{eta},
\end{equation}
so that (by definition of $\Phi_{4, \Tstar}$)
\begin{equation} \label{lbd-v}
\inf  _{ 0\le t\le \Tstar }  \Bigl\{  \eta(t)\inf\limits_{x\in\R^N}\la x\ra^n|v(t,x)|  \Bigr\} \geq1 .
\end{equation}
If $2\leq|\beta|\leq2m$, we deduce from~\eqref{eta}, \eqref{uppbd-2} and~\eqref{bd-T*}  that
\begin{equation}\label{est-eta1}
\begin{split}
(\eta\|v\|_{1,|\beta|-1}) & ^{2|\beta|}\|v  \|_{1,|\beta|-1} \\
&\leq (4K(1-bt)^{-\sigma_1}4K(1-bt)^{-\sigma_{|\beta|-1}})^{2|\beta|}4K(1-bt)^{-\sigma_{|\beta|-1}}\\
&\leq (4K)^{4|\beta|+1}(1-bt)^{-(\sigma_1+\sigma_{|\beta|-1})2|\beta|-\sigma_{|\beta|-1}}\\
&\leq (4K)^{8m+1}(1-bt)^{-\sigma_{|\beta|}},
\end{split}
\end{equation}
since $(\sigma_1+\sigma_{|\beta|-1})2|\beta|+\sigma_{|\beta|-1}\leq(8m+1)\sigma_{|\beta|-1}=\sigma_{|\beta|}$ by~\eqref{de-sigma}. Similarly,
\begin{equation}\label{est-eta2}
\begin{split}
(\eta\|v\|_{1,2m}) ^{2J+\alpha}\|v\|_{1,2m}
&\leq (4K)^{4J+2\alpha+1}(1-bt)^{-(2J+\alpha)(\sigma_1+\sigma_{2m})-\sigma_{2m}}\\
&\leq (4K)^{4J+2\alpha+1}(1-bt)^{-(4J+2\alpha+1)\sigma_{2m}}\\
&=(4K)^{4J+2\alpha+1}(1-bt)^{\frac{2-N\alpha}{2}-\sigma_{2m+1}},
\end{split}
\end{equation}
where the last equality follows from the definition of $\sigma_{2m+1}$ by~\eqref{de-sigma}.
As well, if $2m+2\leq|\beta|\leq J$ and $\ell \in \{ 2, 3 \}$, then
\begin{equation}\label{est-eta3}
\begin{split}
(\eta\|v\|_{1,2m})  ^{2J}\|v\|_{ \ell ,|\beta|-1}
&\leq (4K)^{4J+1}(1-bt)^{-2J(\sigma_1+\sigma_{2m})-\sigma_{|\beta|-1}}\\
&\leq (4K)^{4J+1}(1-bt)^{-4J\sigma_{2m}-\sigma_{|\beta|-1}}\\
&=(4K)^{4J+1}(1-bt)^{-\sigma_{|\beta|}}
\end{split}
\end{equation}
where we used $2J(\sigma_1+\sigma_{2m})+\sigma_{|\beta|-1}\leq4J\sigma_{2m}+\sigma_{|\beta|-1}=\sigma_{|\beta|}$ by~\eqref{de-sigma}.

Since $\|v\|_{L^\infty}\leq 4K$ by~\eqref{uppbd-2}, it follows from~\eqref{fEIn1}  that, given any $\sigma >0$ and $0\le t< \frac {1} {b}$, 
\begin{equation*} 
\begin{split} 
\int _0^{t} ( 1-bs)^{-\frac{4-N\alpha}{2}-\sigma } \|v (s) \|_{L^\infty}^\alpha ds & \le \frac {2 (4K)^\alpha } {b (2-N\alpha + 2 \sigma ) } (1-bt )^{- \frac {2-N\alpha } {2} -\sigma } \\
&\le \frac { (4K)^\alpha } {b  \sigma  } (1-bt )^{- \frac {2-N\alpha } {2} -\sigma } .
\end{split} 
\end{equation*} 
Let $0< t' < \frac {1} {b}$ be defined by $(1-bt' )^{ \frac {2-N\alpha } {2}  } = \frac {1} {2}$,
i.e. $(1-bt')= 2^{- \frac {2} {2-N\alpha }}$. It follows from the above inequality that if $0\le t\le t'$, then
\begin{equation} \label{fStma1} 
\begin{split} 
\int _0^{t} ( 1-bs)^{-\frac{4-N\alpha}{2}-\sigma } \|v (s) \|_{L^\infty}^\alpha ds \le \frac { 2 (4K)^\alpha } { b  \sigma  } (1-bt )^{- \sigma } .
\end{split} 
\end{equation} 
Moreover, if $b\ge b_0$ and $t'\le  t< \Tstar $, then it follows from~\eqref{Decay-v} and~\eqref{fEIn1}  that 
\begin{equation} \label{fStma3} 
\begin{split} 
\int _{t'}^t ( 1-bs)^{-\frac{4-N\alpha}{2}-\sigma } \|v (s) \|_{L^\infty}^\alpha ds & \le 
\frac{b(2-N\alpha)}{2\alpha| \Re \lambda|} \int _{t'}^t  \frac{(1-bs)^{-1 -\sigma } } {1-(1-bs)^\frac{2-N\alpha}{2}} \\
& \le \frac{ 2-N\alpha }{ \sigma \alpha| \Re \lambda|}  (1-bt)^{-\sigma } .
\end{split} 
\end{equation} 
Using $\int _0^t = \int _0^{t'} + \int  _{ t' }^t$ if $t'<t< \Tstar$, we deduce  from~\eqref{fStma1} and~\eqref{fStma3} that for all $b\ge b_0$ and all $0\le t< \Tstar$,
\begin{equation} \label{est-int}
\int _0^{t} ( 1-bs)^{-\frac{4-N\alpha}{2}-\sigma } \|v (s) \|_{L^\infty}^\alpha ds \le  \Bigl(  \frac {2 (4K)^\alpha } {b  \sigma  } +  \frac{ 2-N\alpha }{ \sigma \alpha| \Re \lambda|}  \Bigr) (1-bt)^{-\sigma } .
\end{equation} 
Now, we are ready to estimate $\Psi _{ \Tstar }$ and the process is divided into four steps. We first estimate $\|\la x\ra^nv\|_{L^\infty}$. Since $\Re \lambda<0$, it follows from~\eqref{eq-v-1} and~\eqref{fDfnL}  that
\begin{equation*} 
\partial _t |v| \leq|L| \le  | \Delta v |,
\end{equation*} 
so that 
\begin{equation*} 
\begin{split} 
\la x\ra^n|v(t)| - \la x\ra^n|v_0| \le  \int^t_0\la x\ra^n|\Delta v|\,ds \le \int^t_0\|v\|_{ \Spa }\,ds \leq  \frac{8K \CST3}{b },
\end{split} 
\end{equation*} 
where we used~\eqref{bd-inter-v} in the last inequality.
Since $ \la x\ra^n|v_0| \le K $, we deduce that if $b \ge b_1 $ with $b_1$ given by~\eqref{fDfnbun}, then
\begin{equation}\label{est-de-0}
\|\la x\ra^nv\|_{L^\infty} \leq 2K.
\end{equation}
We next estimate $\|\la x\ra^nD^\beta v\|_{L^\infty}$ for $1\leq|\beta|\leq2m$. 
Applying~\eqref{est-solu-small-der} and~\eqref{fEstdpt1}, we obtain
\begin{equation} \label{fOOMM1} 
\|\la x\ra^nD^\beta v\|_{L^\infty}  \leq 2K  +|\lambda| \CST1\int^t_0(1-bs)^{-\frac{4-N\alpha}{2}}  \|\la x\ra^nD^\beta(|v|^\alpha v)\|  _{ L^\infty  } ds .
\end{equation} 
Using~\eqref{lbd-v}, \eqref{nonlin-est-1}-\eqref{nonlin-est-2}, \eqref{uppbd-2} and~\eqref{est-eta1}
and setting  $\kappa  =0$ if $ |\beta |=1$ and $\kappa (\beta ) =1$ if $ |\beta |\ge 2$, we see that
\begin{equation}  \label{fOOMM2} 
\begin{split} 
\|  \langle \cdot \rangle ^n  D^{\beta} ( |v|^\alpha v)   \|  _{L^{\infty}}  \le &  \CST2 \| v \| _{ L^\infty  }^\alpha  \| \langle \cdot\rangle ^{n}D^{\beta} v \| _{L^\infty} \\ &    + \kappa  \CST2  \| v \| _{ L^\infty  }^\alpha  ( \eta \|  v\| _{1, | \beta |  -1})^{2|  \beta|  }\|  v\|  _{1,  |  \beta|  -1} \\ \le & 2 \CST2 (4K)^{8m +1} \| v \| _{ L^\infty  }^\alpha (1-bs)^{-\sigma  _{  |\beta | }}  .
\end{split} 
\end{equation}
Applying now~\eqref{est-int} and using $\sigma  _{  |\beta | }\ge \sigma _1$, we deduce from~\eqref{fOOMM1}-\eqref{fOOMM2}  that
\begin{equation*} 
\|\la x\ra^nD^\beta v\|_{L^\infty}  \leq  2K  + 2  (4K)^{8m+1} |\lambda | \CST1 \CST2  \Bigl(  \frac {2 (4K)^\alpha } {b  \sigma  _1 }  +  \frac{ 2-N\alpha }{ \sigma  _1 \alpha| \Re \lambda|} \Bigr) (1-bt)^{-\sigma  _{  |\beta | } } .
\end{equation*} 
It follows, using also~\eqref{fDfnalphaun}, \eqref{fDfnbun} and~\eqref{est-de-0}, that
\begin{equation} \label{est-phi1}
\Phi_{1,\Tstar} \le 3K .
\end{equation} 

We next estimate similarly $\|\la x\ra^nD^\beta v\|_{L^2}$ for $2m+1\leq|\beta|\leq2m+2+k$. It follows from~\eqref{est-solu-lar-der} (with $\mu=0$) and~\eqref{fEstdpt1} that
\begin{equation*} 
\|\la x\ra^nD^\beta v\|_{L^2} \le 2K + 
\CST1|\lambda|\int^t_0(1-bs)^{-\frac{4-N\alpha}{2}}\|\la x\ra^nD^\beta(|v|^\alpha v)\|_{L^2}\,ds .
\end{equation*} 
Using~\eqref{nonlin-est-3}, \eqref{bd-T*}, \eqref{est-eta2} and~\eqref{est-eta3}, we see that
\begin{equation*}
\begin{split} 
\|\la x\ra^nD^\beta(|v|^\alpha v)\|_{L^2} \leq &  \CST2 (4K) (1-bs)^{- \sigma  _{  |\beta | }} \|v\|_{L^\infty}^\alpha \\ &
+\CST2 (4K)^{4J+2\alpha+1}(1-bs )^{\frac{2-N\alpha}{2}-\sigma_{2m+1}} \\ &+ \CST2 (4K)^{4J+1}(1-b s )^{-\sigma_{|\beta|}} \|v\|_{L^\infty}^\alpha ,
\end{split} 
\end{equation*} 
so that
\begin{equation*} 
\begin{split} 
\|\la x\ra^nD^\beta v\|_{L^2} \le  & 2K + 
\CST1 \CST2 (4K)^{4J+2\alpha+1} |\lambda|\int^t_0(1-bs)^{-1 - \sigma  _{ 2m+1 }} ds \\
&+ 
\CST1 \CST2 (4K)^{4J+2} | \lambda|\int^t_0(1-bs)^{-\frac{4-N\alpha}{2} - \sigma  _{  |\beta | }} \|v\|_{L^\infty}^\alpha  ds .
\end{split} 
\end{equation*} 
Applying~\eqref{fEIn1} and~\eqref{est-int} to estimate the integrals, we obtain
\begin{equation} \label{fEstInterm1} 
\begin{split} 
\|\la x\ra^nD^\beta v\|_{L^2} \le  & 2K + 
 \frac { \CST1 \CST2 (4K)^{4J+2\alpha+1} |\lambda| } {b \sigma  _{ 2m+1 }}(1-bt)^{-\sigma  _{ 2m+1 }} \\
+ 
\CST1 \CST2 & (4K)^{4J+2} | \lambda|   \Bigl(  \frac {2 (4K)^\alpha } {b  \sigma  _{  |\beta | } }  
 +  \frac{ 2-N\alpha }{ \sigma _{  |\beta | } \alpha| \Re \lambda|} \Bigr) (1-bt)^{-\sigma _{  |\beta | } }  .
\end{split} 
\end{equation} 
Using $\sigma _1\le \sigma  _{ 2m+1 }\le \sigma  _{  |\beta | } $, it follows that
\begin{equation*} 
\begin{split} 
\|\la x\ra^nD^\beta v\|_{L^2} \le  & 2K + 
 \frac { \CST1 \CST2 (4K)^{4J+2\alpha+1} |\lambda| } {b \sigma  _1 }(1-bt)^{-\sigma  _{  |\beta | }} \\
+ 
\CST1 \CST2 & (4K)^{4J+2} | \lambda|   \Bigl(  \frac {2 (4K)^\alpha } {b  \sigma  _1 }   +  \frac{ 2-N\alpha }{ \sigma _1 \alpha| \Re \lambda|} \Bigr) (1-bt)^{-\sigma _{  |\beta | } }  .
\end{split} 
\end{equation*} 
Using also~\eqref{fDfnalphaun} and~\eqref{fDfnbun}, we conclude that
\begin{equation} \label{est-phi2} 
\Phi  _{ 2, \Tstar } \le 3K.
\end{equation} 

Now we estimate $\|\la x\ra^{J-|\beta|}D^\beta v\|_{L^2}$ for $2m+k+3\leq|\beta|\leq J$. 
 It follows from~\eqref{est-solu-lar-der} (with $\mu= n+  |\beta |-J$, $\nu = k+1$) and~\eqref{fEstdpt1} that
\begin{equation*} 
\|\la x\ra^{J-|\beta|}D^\beta v\|_{L^2} \le 2K +  \CST1  |\lambda | \int^t_0  \|\la x\ra^{ J -  |\beta |}D^\beta ( |v|^\alpha v) \|_{L^2} ds .
\end{equation*} 
Using~\eqref{nonlin-est-4}, \eqref{bd-T*}, \eqref{est-eta2} and~\eqref{est-eta3}, we see that
\begin{equation*}
\begin{split} 
\|\la x\ra^{J-|\beta|}D^\beta(|v|^\alpha v)\|_{L^2} \leq &  \CST2 (4K) (1-bs)^{- \sigma  _{  |\beta | }} \|v\|_{L^\infty}^\alpha \\ &
+\CST2 (4K)^{4J+2\alpha+1}(1-bs )^{\frac{2-N\alpha}{2}-\sigma_{2m+1}} \\ &+ 2 \CST2 (4K)^{4J+1}(1-b s )^{-\sigma_{|\beta|}} \|v\|_{L^\infty}^\alpha ,
\end{split} 
\end{equation*} 
so that
\begin{equation} \label{fEstInterm2} 
\begin{split} 
\|\la x\ra^{ J -  |\beta | } D^\beta v\|_{L^2} \le  & 2K + 
\CST1 \CST2 (4K)^{4J+2\alpha+1} |\lambda|\int^t_0(1-bs)^{-1 - \sigma  _{ 2m+1 }} ds \\
&+ 
\CST1 \CST2 (4K)^{4J+2} | \lambda|\int^t_0(1-bs)^{-\frac{4-N\alpha}{2} - \sigma  _{  |\beta | }} \|v\|_{L^\infty}^\alpha  ds .
\end{split} 
\end{equation} 
The right-hand side of~\eqref{fEstInterm2} is similar to the right-hand side of~\eqref{fEstInterm1}, and we conclude as above that
\begin{equation} \label{est-phi3}
\Phi_{3,\Tstar} \le 3K .
\end{equation} 

Finally, we estimate $\Phi_{4, \Tstar}$, we we set
\begin{equation*} 
w (t,x)=\la x\ra^n|v(t,x)|.
\end{equation*} 
Multiplying~\eqref{eq-v-2} by $\la x\ra^{-n\alpha}$ and integrating in $t$, we obtain
\begin{equation*} 
\begin{split} 
\frac{1}{| w (t,x)|^\alpha}= & \frac{1}{| w (0, x)|^\alpha}+\alpha| \Re \lambda|\int^t_0\la x\ra^{-n\alpha}(1-bs)^{-\frac{4-N\alpha}{2}}\,ds\\ &
-\alpha\int^t_0 | w (s,x)|^{-\alpha-1} \la x\ra^{n}L(s,x)\,ds.
\end{split} 
\end{equation*} 
Applying~\eqref{bd-T*}, we see that $ \la x\ra^{n}  |L| \le \la x\ra^{n}  |\Delta v|\le 4K (1-bs)^{-\sigma  _2}$, and  $  |w |^{-\alpha -1}\le (4K)^{\alpha +1} (1-bs)^{- (\alpha +1)\sigma _1} $. Since $(\alpha +1)\sigma _1 + \sigma _2 \le \sigma _3$ by~\eqref{de-sigma}, we deduce that
\begin{equation*} 
\frac{1}{| w (t,x)|^\alpha} \le   K^\alpha  +\alpha| \Re \lambda|\int^t_0 (1-bs)^{-\frac{4-N\alpha}{2}}\,ds 
+ \alpha (4K)^{\alpha +2} \int^t_0 (1-bs)^{- \sigma _3} ds.
\end{equation*} 
Since $ -\frac{4-N\alpha}{2}= -1 -\alpha \sigma _1 + (\alpha \sigma _1 - \frac {2-N\alpha } {2}) $, and since by~\eqref{fDfnalphaun} 
$\alpha  \sigma _1 \ge  \frac {2-N\alpha } {2} $,
we see that $ -\frac{4-N\alpha}{2} \ge  -1 -\alpha \sigma _1 $; and so, using~\eqref{fEIn1} and $\sigma _3\le \sigma _J\le \frac {1} {2} $, 
\begin{equation*} 
\begin{split} 
\frac{1}{| w (t,x)|^\alpha} & \le   K^\alpha  +\alpha| \Re \lambda|\int^t_0 (1-bs)^{- 1 - \alpha \sigma _1 }\,ds 
+ \alpha (4K)^{\alpha +2} \int^t_0 (1-bs)^{- \sigma _3} ds \\
& \le   K^\alpha  +  \frac {| \Re \lambda| }{b \sigma _1}(1-bt)^{-\alpha \sigma _1} 
+  \frac {2 \alpha (4K)^{\alpha +2}} {b  } .
\end{split} 
\end{equation*} 
It follows that
\begin{equation*} 
\Phi  _{ 4, \Tstar }^\alpha \le K^\alpha  +  \frac {| \Re \lambda| }{b \sigma _1}
+  \frac {2 \alpha (4K)^{\alpha +2}} {b  } .
\end{equation*} 
Using~\eqref{fDfnbun}, we deduce that for $b\ge b_1$, 
\begin{equation} \label{est-phi4}
\Phi  _{ 4, \Tstar }^\alpha \le 3^{\frac {1} {N}}K^\alpha  \le ( 3 K)^\alpha ,
\end{equation} 
since $\alpha \ge \frac {1} {N}$. 
Estimates~\eqref{est-phi1}, \eqref{est-phi2}, \eqref{est-phi3} and~\eqref{est-phi4} yield
$\Psi_{ \Tstar }\le 3K$, which leads to a contradiction with~\eqref{bd-T*}. This completes the proof.
\end{proof}

\section{Asymptotics for~\eqref{NLS-1}} \label{sAsympt} 
Throughout this section, we assume~\eqref{def-knm}-\eqref{def-J} and we let $ \Spa $ be defined by~\eqref{fSpa1}-\eqref{fSpa2}. 
We describe the asymptotic behavior as $t\to \frac {1} {b}$ of the solutions of~\eqref{NLS-1} given by Proposition~\ref{BD-GWP}. More precisely, we have the following result.

\begin{prop}\label{asy-v}
Suppose $\Re \lambda<0$.
Let $K\ge 1$, let $\alpha _1 \in ( 0 , \frac {2} {N} ) $ be given by~\eqref{fDfnalphaun} 
and  let $b_1>0$ be given by~\eqref{fDfnbun}.
Let  $v_0\in \Spa $ satisfy~\eqref{IV}, and let $v\in C([0,\frac1b), \Spa )$ be the solution of~\eqref{NLS-1} given by Proposition~$\ref{BD-GWP}$. There exist $f_0, \omega_0 \in L^\infty(\R^N)$, with $f_0$ real-valued,  $\|f_0 \|_{L^\infty}\leq\frac12$,  and $\la \cdot \ra^n\omega_0 \in L^\infty(\R^N)$, such that
\begin{equation}\label{asymptotic}
 \|\la \cdot \ra^n (v(t )-\omega_0 \psi(t )e^{ i\theta(t )} ) \|_{L^\infty}\leq C(1-bt)^{\frac {1} {2} }
\end{equation}
for all $t \in [0, \frac {1} {b}) $, where
\begin{equation}\label{def-psi}
\psi(t,x)=\Big( \frac{1+f_0(x)}{1+f_0(x)+\frac{ 2 \alpha| \Re \lambda|}{b (2-N\alpha) }|v_0(x)|^\alpha[(1-bt)^{-\frac{2-N\alpha}{2}}-1]}\Big)^\frac1\alpha
\end{equation}
and
\begin{equation}\label{def-theta}
\theta(t,x) =\frac{ \Im \lambda}{  \Re \lambda}\log( \psi(t,x)).
\end{equation}
Moreover,
\begin{equation} \label{fDscromzero} 
 |\omega _0 |^\alpha = \frac { |v_0 |^\alpha } {1+ f_0},
\end{equation} 
so that $\frac {2} {3}  | v_0|^\alpha  \le  |\omega _0|^\alpha  \le 2 | v_0|^\alpha $.
In addition,
\begin{equation}\label{converge}
(1-bt)^{-\frac{2-N\alpha}{2}}\|v (t) \|_{L^\infty}^\alpha \goto_{ t\uparrow \frac{1}{b} }\frac{b (2-N\alpha ) }{ 2 \alpha| \Re \lambda|}, 
\end{equation}
and 
\begin{equation}\label{decay-l2-v}
(1- b t )^{ (\frac {1} {\alpha } - \frac {N} {2}) ( 1- \frac {N} {2n} )} \lesssim \| v (t)\|_{L^2} \lesssim (1 - b t )^{ (\frac {1} {\alpha } - \frac {N} {2}) ( 1- \frac {N} {2n} )}
\end{equation}
as $t\to \frac {1} {b} $, where $n$ is given by~\eqref{def-knm}.  
\end{prop}

\begin{proof}
We let $f$ be defined by~\eqref{fDfnf1}. It follows from~\eqref{fEstf1} that  $f(t, \cdot )$ is convergent in $L^\infty(\R^N)$ as $t\uparrow\frac1b$. Then $f$ can be extended to a continuous function $[0,\frac1b]\rightarrow L^\infty(\R^N)$ and we set
\begin{equation*}
f_0(x)=f  \Bigl( \frac1b,x  \Bigr) =-\alpha\int^\frac1b_0|v_0(x)|^\alpha|v(s,x)|^{-\alpha-1}L(s,x)\,ds.
\end{equation*}
By using~\eqref{fEstf1}, \eqref{bd-f}, \eqref{bd-f-1} and $\sigma _3\le \sigma _J\le \frac {1} {2} $ (see~\eqref{festsigj2}), we have for all $0\leq t\leq\frac1b$
\begin{align}
\|f(t)-f_0\|_{L^\infty}&\leq\frac14(1-bt)^{1- \sigma _3 }, \label{diff-f-f0}\\
\|f(t)\|_{L^\infty}&\leq\frac14 . \label{bd-f0}
\end{align}
In particular, $1+ f_0 >0$, so that by~\eqref{def-psi},
\begin{equation} \label{fEstPsiinfty} 
0 < \psi (t,x) \le 1
\end{equation}  
for all $0\le t< \frac {1} {b}$ and $x\in \R^N $.
Moreover, it follows from~\eqref{bd-f0} that
\begin{equation}\label{bd-1}
\Big\|\frac{1}{1+f_0(x)+\frac{ 2\alpha| \Re \lambda|}{b ( 2-N\alpha ) }|v_0(x)|^\alpha[(1-bt)^{-\frac{2-N\alpha}{2}}-1]}\Big\|_{L^\infty}\leq2
\end{equation}
for all $0\leq t<\frac1b.$
We set
\begin{equation}\label{def-tilde-v}
\widetilde{v}(t,x)=\Big( \frac{|v_0(x)|^\alpha}{1+f_0(x)+  \frac{ 2 \alpha| \Re \lambda|}{b (2-N\alpha ) }|v_0(x)|^\alpha[(1-bt)^{-\frac{2-N\alpha}{2}}-1]}\Big)^\frac{1}{\alpha}.
\end{equation}
It follows from~\eqref{IV} and~\eqref{bd-1} that
\begin{equation*}
\|\la x\ra^n\widetilde v(t)\|_{L^\infty}\leq 2^\frac1\alpha K,
\end{equation*}
and we deduce from~\eqref{eq-v-4}, \eqref{diff-f-f0} and~\eqref{bd-1} that
\begin{equation}\label{diff-v-tilde-v}
\|\la \cdot \ra^{n\alpha}(|v(t, \cdot )|^\alpha-\widetilde v(t, \cdot )^\alpha)\|_{L^\infty}\leq  \|\la \cdot \ra^n v_0 \|_{L^\infty}^\alpha(1-bt)^{ 1-\sigma _3 }\leq K^\alpha(1-bt)^{ 1- \sigma _3 }
\end{equation}
for all $0\leq t<\frac1b$.
Next, we introduce the decomposition
\begin{equation}\label{decom-v}
v(t,x)=\omega(t,x)\psi(t,x)e^{ i\theta(t,x)},
\end{equation}
where $\psi(t,x)$ and $\theta(t,x)$ are defined by~\eqref{def-psi} and~\eqref{def-theta} respectively. Differentiating~\eqref{decom-v} with respect to $t$, we obtain
\begin{equation} \label{fEqomega} 
\partial _t \omega = \frac{e^{ - i\theta}}{\psi} \partial _t v  - \omega\frac{ \partial _t \psi}{\psi} - i \omega \partial _t \theta .
\end{equation}
On the other hand, it follows easily from~\eqref{def-psi},~\eqref{def-theta} and~\eqref{def-tilde-v} that
\begin{align*}
\frac{ \partial _t \psi }{\psi}&=\Re \lambda(1-bt)^{-\frac{4-N\alpha}{2}}\widetilde v^\alpha,\\
 \partial _t \theta &= \Im \lambda(1-bt)^{-\frac{4-N\alpha}{2}}\widetilde v^\alpha.
\end{align*}
Therefore, we deduce from~\eqref{fEqomega}, \eqref{decom-v} and~\eqref{NLS-1} that
\begin{equation} \label{eq-omega}
\begin{split}
 \partial _t \omega &= \frac{e^{ - i\theta}}{\psi} \partial _t v -\omega(1-bt)^{-\frac{4-N\alpha}{2}}\lambda\widetilde v^\alpha\\
         &=\frac{e^{ - i\theta}}{\psi}(  \partial _t v-\lambda(1-bt)^{-\frac{4-N\alpha}{2}}\widetilde v^\alpha v)\\
         &=\frac{e^{ - i \theta}}{\psi}( i \Delta v+\lambda(1-bt)^{-\frac{4-N\alpha}{2}}(|v|^\alpha-\widetilde v^\alpha)v) . 
\end{split}
\end{equation}
Next, it follows from~\eqref{def-psi} and the property $1+f_0\geq\frac12$ that
\begin{equation} \label{bd-1/psi-0} 
\psi (t, x) ^{-\alpha } \le 1 +  \frac{ 2 \alpha| \Re \lambda|}{b (2-N\alpha ) }|v_0(x)|^\alpha (1-bt)^{-\frac{2-N\alpha}{2}} .
\end{equation} 
Moreover, we deduce from~\eqref{Decay-v} that if $t' \le t< \frac {1} {b}$ where $t'\in (0, \frac {1} {b})$ is defined by $(1-bt' )^\frac{2-N\alpha}{2} = \frac {1} {2}$, then
\begin{equation*}
|v(t,x) | ^\alpha \leq \frac{b(2-N\alpha)}{\alpha| \Re \lambda|} (1-bt)^\frac{2-N\alpha}{2},
\end{equation*}
hence
\begin{equation*}
(1-bt)^{- \frac{2-N\alpha}{2} } \leq \frac{b(2-N\alpha)}{\alpha| \Re \lambda|} |v(t,x) | ^{- \alpha} .
\end{equation*}
Therefore, it follows from~\eqref{bd-1/psi-0} that $\psi (t, x) ^{-\alpha } \le 1 + 2 |v_0(x)|^\alpha |v(t,x) | ^{- \alpha}  $.
Applying~\eqref{uppbd-3}, \eqref{IV} and~\eqref{UB}, we conclude that
\begin{equation} \label{bd-1/psi} 
\psi (t, x) ^{-\alpha } \le C(1-bt)^{-\alpha\sigma_1}
\end{equation} 
for $t'\le t<\frac {1} {b}$. 
It follows from~\eqref{eq-omega} and~\eqref{bd-1/psi} that
\begin{equation*} 
\begin{split} 
\|\la \cdot \ra^n \partial _t \omega \|_{L^\infty} & \\\le C(1- & bt)^{- \sigma_1} [ \|\la \cdot \ra^n \Delta v \|_{L^\infty}
+ (1-bt)^{-\frac{4-N\alpha}{2}}\||v|^\alpha-\widetilde v^\alpha\|_{L^\infty}\|\la \cdot \ra^nv\|_{L^\infty} ] .
\end{split} 
\end{equation*} 
Since $ \|\la \cdot \ra^n\Delta v\|_{L^\infty}\leq 4K(1-bt)^{-\sigma_2}$ and $\|\la \cdot \ra^nv\|_{L^\infty}\leq4K$ by~\eqref{UB}, we deduce using~\eqref{diff-v-tilde-v} that
\begin{equation}\label{est-omega-t}
\|\la \cdot \ra^n  \partial _t \omega \|_{L^\infty}\leq C(1-bt)^{-\sigma_1}\Big[(1-bt)^{-\sigma_2}+(1-bt)^{-\frac{2-N\alpha}{2}-\sigma_3}\Big]\leq C(1-bt)^{ - \frac {1} {2} },
\end{equation}
where we used $\sigma_1+\sigma_2\leq\sigma_1+\frac{2-N\alpha}{2}+\sigma_3\leq\sigma_J\le \frac {1} {2} $ by~\eqref{de-sigma} and~\eqref{festsigj2}.
It follows from~\eqref{est-omega-t} that if $t' \le t < s < \frac {1} {b}$, then
\begin{equation*}
\|\la \cdot \ra^n(\omega(t)-\omega(s))\|_{L^\infty}\leq C(1-bt)^{\frac {1} {2} },
\end{equation*}
so that there exists $\omega_0$ such that $\la x\ra^n\omega_0\in L^\infty(\R^N)$ and
\begin{equation}\label{diff-omega}
\|\la \cdot \ra^n(\omega(t)-\omega_0)\|_{L^\infty}\leq C(1-bt)^{ \frac {1} {2} }
\end{equation}
for all $ t ' \leq t<\frac1b$.
Using~\eqref{decom-v}, \eqref{fEstPsiinfty}  and~\eqref{diff-omega}, we obtain
\begin{equation*}
\|\la \cdot \ra^n(v(t )-\omega_0\psi(t )e^{ i\theta(t )})\|_{L^\infty}\leq\|\la \cdot \ra^n(\omega(t)-\omega_0)\|_{L^\infty}\|\psi\|_{L^\infty}
\leq C(1-bt)^{ \frac {1} {2} },
\end{equation*}
which proves~\eqref{asymptotic}. 

Next, we prove~\eqref{fDscromzero}.  It follows from~\eqref{asymptotic} (recall that $0\le \psi \le 1$) that
\begin{equation*}
 \|  \, |v(t )| - |\omega_0| \psi(t )   \|_{L^\infty}\leq C(1-bt)^{\frac {1} {2} }.
\end{equation*}
Using the elementary inequalities $ |x^\alpha - y^\alpha | \le  |x-y|^\alpha $ if $\alpha \le 1$ and  $ |x^\alpha - y^\alpha | \le \alpha ( x^{\alpha -1}+ y^{\alpha -1) }  |x-y| $ if $\alpha \ge 1$, and the boundedness of $ \|\langle \cdot \rangle ^n v\| _{  L^\infty  }$, we deduce that 
\begin{equation*}
\| \, |v(t, \cdot )|^\alpha- ( |\omega_0| \psi(t ))^\alpha \|_{L^\infty} \le C ( (1-bt)^{\frac {1} {2}}  + (1-bt)^{\frac {\alpha } {2}} )
\end{equation*} 
Moreover, it follows from~\eqref{diff-v-tilde-v} and $\sigma _3\le \frac {1} {2}$  that
\begin{equation*}
\| \, |v(t, \cdot )|^\alpha-\widetilde v(t, \cdot )^\alpha \|_{L^\infty} \leq K^\alpha(1-bt)^{ \frac {1} {2} } .
\end{equation*}
Thus we see that 
\begin{equation*}
\| \widetilde{v} (t, \cdot )^\alpha- ( |\omega_0| \psi(t ))^\alpha \|_{L^\infty} \le C (1-bt)^{\frac {\rho  } {2}}  ,
\end{equation*} 
where $\rho = \min\{ \alpha , 1\}$.
Using the explicit expressions~\eqref{def-psi} and~\eqref{def-tilde-v}, we obtain
\begin{equation*}
 \frac{ |\,  |v_0 (x)|^\alpha  - (1+f_0(x))  |\omega _0 (x) |^\alpha |}{1+f_0(x)+\frac{ 2 \alpha| \Re \lambda|}{b (2-N\alpha) }|v_0(x)|^\alpha[(1-bt)^{-\frac{2-N\alpha}{2}}-1]} \le C (1-bt)^{\frac {\rho  } {2}} .
\end{equation*} 
For $\frac {1} {2b}\le t< \frac {1} {b}$, we have
\begin{equation*} 
1+f_0(x)+\frac{ 2 \alpha| \Re \lambda|}{b (2-N\alpha) }|v_0(x)|^\alpha[(1-bt)^{-\frac{2-N\alpha}{2}}-1] \le  C (1-bt)^{-\frac{2-N\alpha}{2}} ,
\end{equation*} 
so that
\begin{equation*} 
 |\,  |v_0 (x)|^\alpha  - (1+f_0(x))  |\omega _0 (x) |^\alpha | \le C (1-bt)^{ \frac {\rho } {2} -\frac{2-N\alpha}{2}} .
\end{equation*} 
since $\alpha > \min\{ \frac {1} {N}, \frac {2} {N+1}\}$ by~\eqref{fDfnalphaun}, we see that $ \frac {\rho } {2} -\frac{2-N\alpha}{2} >0$. Letting $t\uparrow \frac {1} {b}$ in the above inequality, we obtain~\eqref{fDscromzero}. 

Now, we prove~\eqref{converge}. Set
\begin{equation*} 
Z (t, x)= (1-bt)^{-\frac{2-N\alpha}{2}}  | \omega_0 (x) \psi(t, w )e^{ i\theta(t, x )}|^\alpha  =  (1-bt)^{-\frac{2-N\alpha}{2}}  | \omega_0 (x) \psi(t, w ) |^\alpha  .
\end{equation*} 
It follows from~\eqref{def-psi} and~\eqref{fDscromzero}   that
\begin{equation*}
Z (t, x) =\frac{|v_0 (x) |^\alpha(1-bt)^{-\frac{2-N\alpha}{2}}}{1+f_0 (x) +\frac{ 2 \alpha| \Re \lambda|}{b ( 2-N\alpha ) }|v_0 (x) |^\alpha[(1-bt)^{-\frac{2-N\alpha}{2}}-1]}.
\end{equation*}
Since $1+ f_0\ge 0$ by~\eqref{bd-f0}, we obtain
\begin{equation*}
Z(t,x)  \leq \frac{b (2-N\alpha ) }{ 2 \alpha| \Re \lambda|}\frac{1}{1-(1-bt)^\frac{2-N\alpha}{2}} 
\end{equation*}
so that
\begin{equation*} 
\limsup  _{ t\uparrow \frac {1} {b} }  \| Z(t) \| _{ L^\infty  }  \le \frac{b (2-N\alpha ) }{ 2 \alpha| \Re \lambda|} . 
\end{equation*} 
Moreover, $1+ f_0 \le 2$, so that
\begin{equation*}
Z (t,0)\geq\frac{|v_0(0)|^\alpha(1-bt)^{-\frac{2-N\alpha}{2}}}{2+\frac{ 2 \alpha| \Re \lambda|}{b (2-N\alpha ) }|v_0(0)|^\alpha[(1-bt)^{-\frac{2-N\alpha}{2}}-1]}
\end{equation*}
Since $|v_0(0)|>0$ by~\eqref{IV}, we deduce that
\begin{equation*}
\liminf_{t\uparrow\frac1b}   \| Z(t) \| _{ L^\infty  }   \ge \frac{b (2-N\alpha ) }{ 2 \alpha| \Re \lambda|} .
\end{equation*}
Thus we see that $ \| Z(t) \| _{ L^\infty  }   \to \frac{b (2-N\alpha ) }{ 2 \alpha| \Re \lambda|}$ as $t\uparrow \frac{1}{b} $. On the other hand, it follows from~\eqref{diff-v-tilde-v} and~\eqref{fDscromzero}  that 
\begin{equation*} 
| (1-bt)^{-\frac{2-N\alpha}{2}}   \|v (t) \|_{L^\infty}^\alpha -  \| Z  (t) \|_{L^\infty}  | \le K^\alpha(1-bt)^{ \frac {N\alpha } {2}- \sigma _3  }.
\end{equation*} 
Since $\frac {N\alpha } {2}\ge \frac {1} {2} > \sigma _3$, \eqref{converge} follows.

Finally, we prove~\eqref{decay-l2-v}. 
It follows from~\eqref{def-psi} and~\eqref{fDscromzero}  that
\begin{equation*} 
 |\omega _0 (x) |^2 \psi (t,x)^2
 =\Big( \frac{|v_0(x)|^\alpha}{1+f_0(x)+  \frac{ 2 \alpha| \Re \lambda|}{b (2-N\alpha ) }|v_0(x)|^\alpha[ (1-bt)^{-\frac{2-N\alpha}{2}}-1]}\Big)^\frac{2}{\alpha}.
\end{equation*} 
Recall that $\frac {1} {2} \le 1+ f_0 \le \frac {3} {2}$ and $\frac {1} {K} \langle x\rangle ^{-n} \le  |v_0 (x) | \le K \langle x\rangle ^{-n}$.  Therefore, for $\frac {1} {2b}\le t< \frac {1} {b}$, we have
\begin{equation} \label{fEstInter1} 
 \frac { a \langle x\rangle ^{-2n}  } { ( 1 + (1-bt)^{-\frac{2-N\alpha}{2}} \langle x\rangle ^{-n \alpha } )^{\frac {2} {\alpha }} }\le  |\omega _0 |^2 \psi^2 \le  \frac { A \langle x\rangle ^{-2n}  } { ( 1 + (1-bt)^{-\frac{2-N\alpha}{2}} \langle x\rangle ^{-n \alpha } )^{\frac {2} {\alpha }} } ,
\end{equation} 
for some constants $0<a \le A <\infty $. 
If $ |x| \ge (1-bt)^{- \frac {2- N\alpha } {2n\alpha }}$, then $ |\omega _0|^2 \psi ^2 \gtrsim  |x|^{-2n}$ by the first inequality in~\eqref{fEstInter1}.
Since also $ |\omega _0|^2 \psi ^2 \lesssim  |x|^{-2n}$ by the second inequality in~\eqref{fEstInter1}, we deduce that
\begin{equation*} 
a_1  (1 - b t )^{ (\frac {2} {\alpha } - N ) ( 1- \frac {N} {2n} )} \le \int  _{  |x| > (1-bt)^{- \frac {2- N\alpha } {2n\alpha }}  }  |\omega _0|^2 \psi ^2 \le A_1 (1 - b t )^{ (\frac {2} {\alpha } - N ) ( 1- \frac {N} {2n} )} ,
\end{equation*} 
for some constants $0<a_1 \le A_1 <\infty $. 
If $ |x| \le  (1-bt)^{- \frac {2- N\alpha } {2n\alpha }}$, then $ |\omega _0|^2 \psi ^2 \gtrsim (1-bt)^{ \frac {2- N \alpha } {\alpha } }$ by the first inequality in~\eqref{fEstInter1}.
Since also $ |\omega _0|^2 \psi ^2 \lesssim (1-bt)^{ \frac {2- N\alpha } {\alpha } }$ by the second inequality in~\eqref{fEstInter1}, we deduce that
\begin{equation*} 
a_2  (1 - b t )^{ (\frac {2} {\alpha } - N ) ( 1- \frac {N} {2n} )} \le \int  _{  |x| < (1-bt)^{- \frac {2- N\alpha } {2n\alpha }}  }  |\omega _0|^2 \psi ^2 \le A_2 (1 - b t )^{ (\frac {2} {\alpha } - N ) ( 1- \frac {N} {2n} )} ,
\end{equation*} 
for some constants $0<a_2 \le A_2 <\infty $. It follows that 
\begin{equation} \label{fEstInter2} 
a_3  (1 - b t )^{ (\frac {1} {\alpha } - \frac {N} {2} ) ( 1- \frac {N} {2n} )} \le  \| \omega _0 \psi (t) e^{ i\theta(t )} \| _{ L^2  } \le A_3  (1 - b t )^{ (\frac {1} {\alpha } - \frac {N} {2} ) ( 1- \frac {N} {2n} )} ,
\end{equation} 
for some constants $0<a_3 \le A_3 <\infty $.
On the other hand, estimate~\eqref{asymptotic} implies (since $n > \frac {N} {2}$)
\begin{equation} \label{fEstInter3} 
 \| v(t )-\omega_0 \psi(t )e^{ i\theta(t )}  \|_{L^2} \leq C(1-bt)^{\frac {1} {2} } .
\end{equation}
Since $\alpha >\frac {2} {N+1}$, we have
\begin{equation*} 
\Bigl(\frac {1} {\alpha } - \frac {N} {2} \Bigr) \Bigl( 1- \frac {N} {2n} \Bigr) < \frac {1} {\alpha } - \frac {N} {2} < \frac {1} {2}
\end{equation*} 
and~\eqref{decay-l2-v} follows from~\eqref{fEstInter2}-\eqref{fEstInter3}.  
\end{proof}

\section{Proof of Theorem~$\ref{ethm1}$} \label{sPrfThm1} 
Let $ \Re \lambda<0$ and $K\geq1$, and let $v_0\in \Spa $ satisfy~\eqref{IV}. Let $\alpha_1$ and $b_1$ be given by Proposition~\ref{BD-GWP}. Given $\alpha_1 \leq\alpha<\frac2N$ and $b\geq b_1$, let $v\in C([0,1/b), \Spa )$ be the corresponding solution of~\eqref{NLS-1} given by Proposition~\ref{BD-GWP}. Let 
\begin{equation} \label{fPseudoConf2} 
u(t,x)=(1+bt)^{-\frac N2}e^{i\frac{b|x|^2}{4(1+bt)}}v\Big( \frac t{1+bt},\frac x{1+bt}\Big),\  t\geq0,\ x\in\R^N .
\end{equation}
It follows that  $u\in C([0,\infty), H^1(\R^N))$ is the solution of~\eqref{NLS-0}  with the initial condition $u_0(x)=e^{i\frac{b|x|^2}{4}}v_0(x)$.
Since $n>\frac N2$, we deduce from~\eqref{asymptotic} in Proposition~\ref{asy-v} that
\begin{equation*}
\|v(t,x)-\omega_0(x)\psi(t,x)e^{ i\theta(t,x)}\|_{L^\infty\cap L^2}\leq C(1-bt)^{\frac {1} {2} }.
\end{equation*}
This proves~\eqref{asy-u}, while~\eqref{decay-u} and~\eqref{decay-l2-u} follow from~\eqref{converge} and~\eqref{decay-l2-v}, respectively. This completes the proof of Theorem~\ref{ethm1}.

\end{document}